\documentclass[11pt, a4paper]{amsart}
\usepackage{amssymb, amsmath, amsthm, amsfonts, mathtools, multicol, inputenc}

\mathtoolsset{showonlyrefs=true}

\usepackage{fancyhdr}
\usepackage[a4paper,left=25mm,right=25mm,top=30mm,bottom=30mm]{geometry}

\usepackage{mathrsfs}
\usepackage{graphicx}
\usepackage{color}
\usepackage[normalem]{ulem}
\usepackage{cancel}
\usepackage{tikz}
\usetikzlibrary{matrix}
\usepackage[all]{xy}
\usepackage{ytableau,varwidth}
\usepackage{pgfplots,caption}
\usepackage{tikz}

\newtheorem{theorem}{Theorem}[section]
\newtheorem{corollary}[theorem]{Corollary}
\newtheorem{lemma}[theorem]{Lemma}
\newtheorem{proposition}[theorem]{Proposition}

\theoremstyle{definition}

\theoremstyle{definition}
\newtheorem{definition}[theorem]{Definition}
\newtheorem{remark}[theorem]{Remark}

\usepackage[normalem]{ulem}

\newcommand{\Z}{\mathbb Z}

\newcommand{\N}{\mathbb N}

\newcommand{\Hom}{\operatorname{Hom}}
\newcommand{\End}{\operatorname{End}}

\newcommand{\rad}{\operatorname{rad}}
\newcommand{\GL}{\operatorname{GL}}

\newcommand{\Tr}{\operatorname{Tr}}

\newcommand{\Irr}{\operatorname{Irr}}

\newcommand{\Aut}{\operatorname{Aut}}
\newcommand{\Out}{\operatorname{Out}}

\newcommand{\PSL}{{\operatorname{PSL}}}
\newcommand{\SL}{{\operatorname{SL}}}
\newcommand{\PGL}{{\operatorname{PGL}}}

\newcommand{\V}{\mathrm{V}}

\keywords{Zassenhaus conjecture, Prime Graph Question, integral group rings, blocks of Klein four defect, unit groups}
\subjclass[2010]{16U60, 20C05, 20C11}

\title{Units in group rings and blocks of Klein four or dihedral defect}

\author{Florian Eisele}
\address{Department of Mathematics, University of Manchester, Oxford Road, Manchester, M13 9PL, United Kingdom.}
\email{florian.eisele@manchester.ac.uk}

\author{Leo Margolis}
\address{Universidad Auton\'oma de Madrid, Departamento de Matem\'aticas, Facultad de Ciencias, Campus Cantoblanco, 28049 Madrid, Spain.}
\email{leo.margolis@icmat.es}

\renewcommand{\leq}{\leqslant}
\renewcommand{\geq}{\geqslant}

\usepackage[hidelinks]{hyperref}
\usepackage{stix}

\usepackage{geometry}

\begin{document}

\maketitle

\begin{abstract}
We obtain restrictions on units of even order in the integral group ring $\mathbb{Z}G$ of a finite group $G$ by studying their actions on the reductions modulo $4$ of lattices over the $2$-adic group ring $\mathbb{Z}_2G$. This improves the ``lattice method'' which considers reductions modulo primes $p$, but is of limited use for $p=2$ essentially due to the fact that $1\equiv -1 \ (\textrm{mod }2)$. Our methods yield results in cases where $\Z_2 G$ has blocks whose defect groups are Klein four groups or dihedral groups of order $8$. This allows us to disprove the existence of units of order $2p$ for almost simple groups with socle $\PSL(2,p^f)$ where $p^f\equiv \pm 3 \ (\textrm{mod } 8)$ and to answer  the Prime Graph Question affirmatively for many such groups. 
\end{abstract}

\section{Introduction}

The structure of the unit group of the integral group ring $\mathbb{Z}G$ of a finite group $G$ has been studied extensively since the 1950's (see \cite{Sehgal93, MargolisDelRioSurvey} for surveys of the area). Many of the strongest questions turned out to have negative answers, among them the Zassenhaus conjecture which was the strongest conjectural assertion made about individual torsion units \cite{EiseleMargolisZassenhaus}. Yet there still are observed properties of the unit group that remain mysterious. This concerns in particular the arithmetic properties of units of finite order. Denote by $V(\mathbb{Z}G)$ the group of units of augmentation $1$ in $\mathbb{Z}G$, also known as normalized units, i.e. the units coefficient sum $1$.\\

\paragraph*{\textbf{Prime Graph Question}} Let $G$ be a finite group. If $u \in V(\mathbb{Z}G)$ has order $pq$, for $p$ and $q$ primes, does it follow that $G$ contains an element of order $pq$?\\

In contrast to other problems in the field, there is a reduction theorem for the Prime Graph Question. Namely, it has a positive answer for a group $G$ if it has a positive answer for all almost simple quotients of $G$ \cite[Theorem 2.1]{KimmerleKonovalov}. Recall that a group $G$ is called \textit{almost simple} if there is a non-abelian simple group $S$ such that $S \leq G \leq \Aut(S)$. In this case $S$ is known as the \textit{socle} of $G$. 

Hence there is hope for a positive answer to the Prime Graph Question using the classification of finite simple groups. Some steps in this direction have been taken, the first one being a proof for groups of the form $\PSL(2,p)$ with $p$ a prime \cite{LutharPassi, HertweckBrauer}. More recent results answered the Prime Graph Question affirmatively for almost simple groups whose socle is an alternating group \cite{BaechleMargolisAn}, for most sporadic groups and some particular series of simple groups of Lie type \cite{CaicedoMargolis}. Nevertheless, much remains to be done and the present paper is a contribution to the study of this question. Namely, we prove:

\begin{theorem}\label{th:Order2p}
Let $q$ be a power of a prime $p$ and let $G$ be an almost simple group with socle $\PSL(2,q)$ such that $q\equiv \pm 3 \ (\textrm{mod } 8)$. Then $V(\mathbb{Z}G)$ contains no unit of order $2p$.  
\end{theorem}

This has the following consequence for the Prime Graph Question:


\begin{theorem}\label{th:ApplicationPQ}
Let $G$ be an almost simple group with socle $\PSL(2,q)$ such that $q\equiv \pm 3 \ (\textrm{mod } 8)$. If $q$ is a prime or the odd part of $(q-1)(q+1)$ is square free, then the Prime Graph Question has a positive answer for $G$.
\end{theorem}

Note that the condition $q\equiv \pm 3 \ (\textrm{mod } 8)$ is equivalent to a Sylow $2$-subgroup of $\PSL(2,q)$ having order~$4$.
Most results on the Prime Graph Question and, more generally, torsion units in $V(\mathbb{Z}\PSL(2,q))$ are based either on a character theoretic approach known as the ``HeLP method'', or the ``lattice method'' introduced formally in \cite{BachleMargolisLattice}, but already present in \cite{HertweckA6}. Roughly speaking the ``lattice method'' works by taking a ``global'' unit, that is, a unit in $\mathbb{Z}G$, of order divisible by $p$, and then projecting it down to a block of the $p$-adic group ring $\mathbb{Z}_pG$. One can then derive restrictions on the eigenvalues of $u$ on the irreducible ordinary representations of $G$, and therefore its ``partial augmentations'', by reducing a lattice in the ordinary representation modulo $p$ and studying the action of the unit on the simple composition factors of this reduction. In all applications of this method so far $p$ was always assumed to be odd, essentially because $1$ and $-1$ cannot be distinguished modulo $2$. In the present paper we apply the idea of the lattice method for the first time in the situation where $p=2$ and the unit has even order. The key idea is to reduce the lattice not modulo $2$ but rather modulo $4$, which does however in practice require a much finer understanding of the structure of the block than in the odd prime case. 
The structure theory available for blocks with a Klein four defect group or a dihedral defect group of order $8$ comes to our aid here. For most blocks relevant to the Prime Graph Question it allows us to very explicitly describe the parts of the block on which the non-trivial $2'$-part of the unit in question acts non-trivially. This in turn gives us information about the action of the unit on the irreducible ordinary representations, leading to the conclusion that the unit either needs to be rationally conjugate to an element of $G$ or does not exist.

Among non-solvable groups the groups of the form $\PSL(2, p^f)$ are certainly the ones for which the units in $\Z G$ have been studied the most \cite{LutharPassi, WagnerDiplom, Bleher95, BleherHissKimmerle95, HertweckBrauer, HertweckA6, HertweckHoefertKimmerle, Gildea13, BaechleMargolisPSL2p3,MargolisPSL, delRioSerrano17, BachleMargolisLattice, BachleMargolis4primaryI, BachleMargolis4primaryII, MargolisdelRioSerrano}. They include the only non-abelian simple groups for which the Zassenhaus conjecture is known to hold (see \cite{MargolisDelRioSurvey} for an overview and \cite[Corollary 1.3]{EiseleMargolisDefect1} for a new result). Yet the only almost simple groups with socle $\PSL(2,p^f)$ for which the Prime Graph Question was known to have a positive answer were those where $f \leq 2$ or where $p=2$ and $f$ satisfies certain restrictions. In particular, there were no positive results for $p$ odd and $f\geq 3$. The main obstacle here were units of order divisible by $p$, the defining characteristic of the group, as highlighted already in the final questions in \cite{HertweckHoefertKimmerle}. Our Theorem~\ref{th:Order2p} makes headway on this problem. However, in some cases it remains to prove the non-existence of units of order $pr$ for certain odd primes $r$. This would remove the condition in Theorem~\ref{th:ApplicationPQ} that the odd part of $(q-1)(q+1)$ has to be square free.

\section{Basic facts and notation}

Throughout this article $G$ denotes an arbitrary finite group.
Given $n\in \N$ we write $\zeta_n$ to denote some primitive complex $n$-th root of unity. For a prime $p$ the $p$-adic integers are denoted $\mathbb{Z}_p$. Throughout the article $R$ denotes an unramified extension of $\Z_2$, and $K$ denotes its field of fractions. Set $F=R/2R$. We are interested in reductions of lattices modulo $4$, so we set $\bar R=R/4R$. In the same vein, if $L$ is an $R$-lattice or an $RG$-lattice we define $\bar L=L/4L$.

\subsection{Torsion units in integral group rings}\label{section:torsionunits}


For $x \in G$ we denote by $x^G$ the conjugacy class of $x$ in $G$. Moreover, for $u = \sum_{g \in G}z_g g \in \mathbb{Z}G$ and $x \in G$ we let 
\[\varepsilon_x(u) = \sum_{g \in x^G} z_g \]
be the \emph{partial augmentation} of $u$ at $x$. Partial augmentations are a key notion in the study of torsion units. The following well-known facts will be used without further mention.

\begin{theorem}\label{th:pAsofTorsionUnits}
Let $u \in \V(\mathbb{Z}G)$ be of finite order $n$.
\begin{itemize}
\item[(i)] If $u \neq 1$, then $\varepsilon_1(u) = 0$ (Berman-Higman Theorem) \cite[Proposition 1.5.1]{GRG1}.
\item[(ii)] Let $g\in G$ such that $\varepsilon_g(u) \neq 0$. Then the order of $g$ divides $n$ \cite[Proposition 2.2]{HertweckBrauer}.
\end{itemize}
\end{theorem}

There are some known congruences for the partial augmentations of torsion units (which often hold more generally, cf. \cite{BovdiMaroti}). We will need the following special form:  
 
\begin{lemma}\label{lem:Congruence}\cite[Lemma 2.2]{BaechleMargolisAn}
Let $p$ be a prime and $u \in V(\mathbb{Z}G)$ a torsion unit of order different from $p$. Let $g_1,\ldots,g_m$ be representatives of the conjugacy classes of elements of order $p$ in $G$. Then 
\[\sum_{i=1}^m \varepsilon_{g_i}(u) \equiv 0 \ (\textrm{mod } p). \]
\end{lemma}

If $\alpha: G \rightarrow \GL(n,K)$ is a representation of $G$ over $K$, then $\alpha$ extends to a ring homomorphism $\mathbb{Z}G \rightarrow M_n(K)$ which in turn restricts to a group homomorphism $V(\mathbb{Z}G) \rightarrow \GL(n,K)$. By abuse of notation we will also denote this latter group homomorphism by $\alpha$. Now if $u \in V(\mathbb{Z}G)$ has order $n$, then the order of $\alpha(u)$ is a divisor of $n$. Since $K$ is a field of characteristic $0$, the matrix $\alpha(u)$ is diagonalizable over an algebraic closure of $K$, with $n$-th roots of unity as its eigenvalues. If $\chi$ is the character belonging to $\alpha$ and $\zeta$ is an $n$-th root of unity in an algebraic closure of $K$, we will let $\mu(\zeta, u, \chi)$ denote the multiplicity of $\zeta$ as an eigenvalue of $\alpha(u)$.

The following formula due to Luthar and Passi is very useful to calculate multiplicities of eigenvalues of torsion units.

\begin{proposition}\cite{LutharPassi}\label{pr:luthar-passi-multiplicity-formula}
   Let $u$ in $\V(\mathbb Z G)$ be of order $n$, denote by $\zeta\in \mathbb C$ an $n$-th root of unity and let $\chi$ be an ordinary character of $G$. Then 
   \[
       \mu(\zeta, u, \chi) =\frac{1}{n}\sum_{d|n} {\rm Tr}_{\mathbb Q(\zeta^d)/\mathbb Q} (\chi(u^d)\zeta^{-d})
   \] 
\end{proposition}

Note that for ordinary characters, their values on torsion units are connected to partial augmentations by the obvious formula
\[\chi(u) = \sum_{g \in \mathcal{C}} \varepsilon_g(u) \chi(g), \]
where $\chi$ denotes any ordinary character and $\mathcal{C}$ a set of representatives of the conjugacy classes in $G$. We will use this formula without further mention.

The next classical result on orders of torsion units allows us to formulate some of the results in a more general way.

\begin{lemma}\label{lem:CohnLivingstone}\cite{CohnLivingstone}
If $u \in V(\mathbb{Z}G)$ has order $n$, then $n$ divides the exponent of $G$.
\end{lemma}

For applications to almost simple groups we will also need the following.

\begin{proposition}\label{prop:KimmerleKonovalovQuotient}\cite[Proposition 2.2]{KimmerleKonovalov}
Let $N$ be a normal subgroup of a finite group $H$ and set $G = H/N$. Assume that the Prime Graph Question has a positive answer for $G$. Then for primes $p$ and $q$ such that $p$ does not divide the order of $N$, the unit group $V(\mathbb{Z}H)$ contains an element of order $pq$ if and only if $H$ contains an element of order $pq$.
\end{proposition}

The following is a consequence of the theory of blocks of defect 1 and the lattice method.

\begin{theorem}\label{th:CaicedoMargolis}\cite[Theorem 1.1]{CaicedoMargolis}
Let $G$ be a finite group, $p$ and $q$ primes and assume a Sylow $p$-subgroup of $G$ has order $p$. Then $V(\mathbb{Z}G)$ contains an element of order $pq$ if and only if $G$ contains an element of order $pq$.
\end{theorem}

Below is a summary of known facts about torsion units of almost simple groups with socle $\PSL(2,p^f)$ which will be useful to us.

\begin{proposition}\label{prop:PSL2KnwonStuff}
Let $G$ be an almost simple group with socle $\PSL(2,p^f)$.
\begin{enumerate}
\item If $f \leq 2$, then the Prime Graph Question has a positive answer for $G$ \cite[Proposition 3.4 and Theorem 3.5]{BachleMargolis4primaryI}.
\item If $G = \PSL(2,p^f)$ or $G = \PGL(2,p^f)$ and $u \in V(\mathbb{Z}G)$ is a torsion unit of order coprime to $p$, then $u$ has the same order as an element of $G$ \cite[Proposition 6.7]{HertweckBrauer}, \cite[Proposition 3.4]{BachleMargolis4primaryI}.
\end{enumerate}
\end{proposition}

\subsection{Structure of blocks with Klein four or dihedral defect groups}

Remark~\ref{rem:Height0AndSplit} below summarizes the essential facts on blocks whose defect groups are Klein four groups or dihedral groups of order eight. Most of this is due to Brauer. We then derive two results partially describing the structure of the blocks we will consider later. While we prove these results directly and elementarily using the methods of Plesken \cite{Plesken}, they can also be derived from existing results. The blocks considered in Proposition~\ref{prop structure v4 block} must be Morita equivalent to $RA_4$ by \cite[Theorem 1.1]{CravenEtAl} (using the classification of finite simple groups), which we can describe directly. The blocks considered in Proposition~\ref{prop structure block d8} were also described in full by the first author in \cite{EiseleDerEq}. As a general note we should point out that there are only a handful of possibilities for the decomposition matrices of blocks of defect $C_2\times C_2$ or $D_8$. There are only three possible decomposition matrices for blocks of Klein four defect, and for dihedral defect $D_8$ the decomposition matrices in Proposition~\ref{prop structure block d8} cover all possibilities for blocks with two isomorphism classes of simple modules. The blocks we consider cover most cases of interest for the Prime Graph Question.



\begin{remark}\label{rem:Height0AndSplit}
    \begin{enumerate}
        \item
        In a block with defect group $C_2 \times C_2$ all irreducible characters have height $0$ by a result of Brauer \cite[Corollary 12.1.5]{LinckelmannVol2}. Moreover, the values of the irreducible characters all lie in $\mathbb{Z}_2[\zeta]$ for $\zeta$ a $2'$-root of unity, as $\chi$ vanishes on elements with $2$-part of order bigger than $2$ by Green's Theorem on Zeros of Characters \cite[Theorem 19.27]{CurtisReinerI}.
        \item A block with defect group $D_8$ has exactly five ordinary irreducible characters, all of which are $2$-rational \cite[Theorem 3]{BrauerDihedralDefect}, meaning they take values in $\mathbb{Z}_2[\zeta]$ for $\zeta$ a $2'$-root of unity. Exactly four of these characters have height $0$, and one  has height $1$. We record for later use that a particular corollary of this is that if the decomposition matrix of the block is as in equation~\eqref{eqn decomp d8} below and the block is principal, then the character belonging to the third row has height $1$ and therefore the first column belongs to the trivial simple module.
        \item As a consequence of the preceding facts, given a block with defect group $C_2\times C_2$ or $D_8$, there exists an unramified extension $R$ of $\Z_2$ such that the block idempotent $b$ is defined over $R$ and both $KGb$ and $FGb$ are split. We can also assume that all Schur indices are $1$ by \cite[IV, Theorem 9.2]{Feit}.
    \end{enumerate}
\end{remark}    

\begin{remark}\label{remark symm form}
    Let $G$ be a finite group. We will repeatedly use the fact that $RG$ is a symmetric order with dualizing form  
    \begin{equation}
        T=\frac{1}{|G|}\chi_{\rm reg}=\frac{1}{|G|}\sum_{\chi\in\Irr(G)}\chi(1)\chi:\ KG \longrightarrow K,
    \end{equation}
    where $\Irr(G)$ denotes the irreducible characters of $G$ over an algebraic closure of $K$.
    Note that if we identify $KG$ with a direct sum of full matrix algebras, then $T$ is explicitly given as a weighted sum of traces (or reduced traces in the non-split case). We will repeatedly use the fact that for idempotents $e,f\in RG$ we have 
    \begin{equation}
        eRGf = \{ a \in eKGf \ | \ T(afRGe) \subseteq R  \},
    \end{equation}
    which means that $eRGf$ is fully determined by $fRGe$ (and vice versa). Even in the case $e=f$ this gives a strong restriction on the shape of $eRGe$.
\end{remark}

\begin{proposition}\label{prop structure v4 block}
    Let $G$ be a finite group and let $RGb$ be a block of defect $C_2 \times C_2$ with decomposition matrix 
    \begin{equation}
    \begin{pmatrix} 1 & 0 & 0 \\ 0 & 1 & 0 \\ 0 & 0 & 1 \\ 1 & 1 & 1 \\  \end{pmatrix}. 
    \end{equation} 
    Assume that $KGb$ is split (which holds for $R$ big enough by Remark~\ref{rem:Height0AndSplit}).
    If $e\in RGb$ is an idempotent annihilating a simple $RGb$-module, then 
    \begin{equation}\label{description order v4 block}
        eRGbe\cong
        \begin{tikzpicture}[baseline={(0,-0.13)}]
            \matrix (mat0) [matrix of math nodes, left delimiter=(, right delimiter=)] {
                R^{a\times a} \\
            };
            \matrix (mat1) [matrix of math nodes, left delimiter=(, right delimiter=)] at (2.5,0) {
                    R^{c\times c} \\
            };
            \matrix (mat2) [matrix of math nodes, left delimiter=(, right delimiter=)] at (5,0) {
                   R^{a\times a} & 4R^{a\times c}\\
                   R^{c\times a} & R^{c\times c}\\
            };
            \draw[thick] (mat0-1-1)  to [bend left]  node[left, above] {$\equiv_4$}  (mat2-1-1);
            \draw[thick] (mat1-1-1) to [bend right]  node[left, below] {$\equiv_4$} (mat2-2-2);
        \end{tikzpicture}
        \subseteq M_a(K)\oplus M_c(K) \oplus M_{a+c}(K)
    \end{equation}
    for certain $a,c\in \N_0$, where the arcs indicate entry-wise congruence modulo $4$.
\end{proposition}
\begin{proof}
    Note that the assertion where $e$ annihilates more than one simple module follows from the assertion where $e$ annihilates exactly one simple module. So we only consider the latter case, and assume without loss of generality that the unique (up to isomorphism) simple module annihilated by $e$ corresponds to the first column of the decomposition matrix.
    
    We will first assume that $eRGbe$ is basic (which corresponds to the case $a=c=1$).  The decomposition matrix then shows that $eKGbe\cong K\oplus K \oplus M_2(K)$. We can fix an isomorphism and regard $eRGbe$ as an order in $K\oplus K \oplus M_2(K)$. As $eRGbe$ has two simple modules which correspond to primitive idempotents $f_1$ and $f_2$ such that $f_1 + f_2 =1$, after conjugation we can assume that $eRGbe$ contains the idempotents $f_1=(1,0,e_{11})$ and $f_2=(0,1,e_{22})$. Here $e_{ij}$ denotes the $(i,j)$ matrix unit in $M_2(K)$. This again follows from the shape of the decomposition matrix. 

    Label the irreducible characters of $G$ belonging to rows of the decomposition matrix above $\chi_1,\chi_2,\chi_3$ and $\chi_4$, in that order. The restriction of the symmetrizing form $T$ from Remark~\ref{remark symm form} to $eKGbe$ maps 
    $(x,y,z)\in K \oplus K \oplus M_2(K)$ to $\frac{\chi_2(1)}{|G|}x+\frac{\chi_3(1)}{|G|}y+\frac{\chi_4(1)}{|G|}{\rm tr}(z)$. Now $f_1eRGbef_1$ has an $R$-basis consisting of $f_1$ and $(0,0,2^te_{11})$ for some $t\in \N$, since any order in $f_1eRGbef_1\cong K\oplus K$ has a basis like this. We know that $f_1eRGbef_1=\{ v\in f_1eKGbef_1 \ | \ T(vf_1eRGbef_1)\subseteq R \}$, which implies that $2^t$ needs to have the same $2$-valuation as $\frac{|G|}{\chi_4(1)}$, that is, $t=2$ (note here that all characters have height zero by Remark~\ref{rem:Height0AndSplit}). Analogously we conclude that  $f_2eRGbef_2$ has an $R$-basis consisting of $f_2$ and $(0,0,4e_{22})$. 

    Note that $f_2eRGbf_1$ has a basis formed by some element $(0,0,2^t e_{21})$ for some $t \in \mathbb{N}_0$ and after conjugation we can ensure this element is $(0,0,e_{21})$. Then $f_1eRGbef_2=\{ v\in f_1eKGbef_2 \ | \ T(vf_2eRGbef_1)\subseteq R \}$, which shows that $f_1eRGbef_2$ has an $R$-basis consisting of $(0,0,4e_{12})$, where ``$4$''  arises as the $2$-part of $\frac{|G|}{\chi_4(1)}$. This completes the calculation of a basis of $eRGbe$, showing that it is of the claimed form.

    If $eRGbe$ is not basic, then we can find an idempotent $e'$ in $eRGbe$ such that $e'RGbe'$ is basic, and therefore of the claimed form with $a=c=1$. But then $eRGbe$ is Morita equivalent to $e'RGbe'$, and the orders described in~\eqref{description order v4 block} as we vary $a$ and $c$ range, up to isomorphism, over a whole Morita equivalence class. To be precise, if $P_1$ and $P_2$ denote the two projective indecomposables over $e'RGbe'$, then the order in \eqref{description order v4 block} is isomorphic to $\End_{e'RGbe'}(P_1^a\oplus P_2^c)$. In particular $eRGbe$ is of this form for some choice of $a$ and $c$. 
\end{proof}

\begin{proposition}\label{prop structure block d8}
    Let $G$ be a finite group and let $RGb$ be a block of defect $D_8$ with decomposition matrix 
    \begin{equation}
    \begin{pmatrix} \label{eqn decomp d8}
        1 & 0 \\ 1 & 0 \\ d & 1 \\ 1 & 1 \\ 1 & 1  \end{pmatrix} \quad \textrm{ for $d\in\{0,2\}$}. 
    \end{equation}
    Assume $KGb$ is split (which holds for $R$ big enough by Remark~\ref{rem:Height0AndSplit}). Let $e\in RGb$ be a non-zero idempotent annihilating the simple module belonging to the first column. Then $eRGbe \cong M_a(\Lambda)\subseteq M_a(K)\oplus M_a(K)\oplus M_a(K)$ for some $a\in \N$, where
    \begin{equation}
         \Lambda = \langle (1,1,1),\ (2,4,0),\ (4,0,0) \rangle_R \subseteq K\oplus K \oplus K.
    \end{equation}
\end{proposition}
\begin{proof}
    As in Proposition~\ref{prop structure v4 block} it suffices to consider the case where $eRGbe$ is basic. The shape of the decomposition matrix then implies that $eRGbe$ is isomorphic to an order in $K\oplus K \oplus K$, which means that it is equal to
    \begin{equation}
        \Lambda=\langle (1,1,1),\ (r2^h, 2^l, 0),\ (2^m, 0,0) \rangle_R
    \end{equation}
    for certain $h,l,m\in \N$ and $r\in R^\times$. The symmetrizing form from Remark~\ref{remark symm form}  maps 
    $(x,y,z)\in K \oplus K \oplus K$ to $\frac{\chi_1(1)}{|G|}x+\frac{\chi_2(1)}{|G|}y+\frac{\chi_3(1)}{|G|}z$, where $\chi_1,\chi_2$ and $\chi_3$ denote the characters belonging to the third, fourth and fifth row of the decomposition matrix. By Remark~\ref{rem:Height0AndSplit} the coefficient $\frac{\chi_1(1)}{|G|}$ has $2$-valuation $-2$, while both $\frac{\chi_2(1)}{|G|}$ and $\frac{\chi_3(1)}{|G|}$ have $2$-valuation $-3$. Considering $T((2^m,0,0)\Lambda)\subseteq R$ and $T((4,0,0)) \in R$ implies that the $m$ is equal to the $2$-valuation of $\frac{|G|}{\chi_1(1)}$, so $m=2$. Considering $T(r2^h, 2^l, 0)= \frac{2^hr\chi_1(1)+2^l\chi_2(1)}{|G|}\in R$ we deduce $h=l-1$ and $r\equiv -\frac{2\chi_2(1)}{\chi_1(1)}\equiv 1\ (\textrm{mod }2)$. It follows that $l\geq 2$, and if $l\geq 3$, then $(2,4,0)$ is not in $\Lambda$ despite satisfying $T((2,4,0)\Lambda)\subseteq R$, a contradiction. So $l=2$ and $h=1$. It follows that the value of $r$ only matters modulo $2$, so we can pick $r=1$ without loss of generality.
\end{proof}

\section{Lifting units modulo $4$}

Now we will use the partial descriptions of the blocks with defect groups $C_2\times C_2$ and $D_8$ given in the preceding section to derive restrictions on the eigenvalues of a unit $u$ on irreducible ordinary representations. The general idea of the lattice method is to restrict irreducible $RG$-lattices to $R[u]$-lattices and the simple $FG$-modules to $F[u]$-modules, and to then use the information we get from the decomposition matrix on filtrations of the former (reduced modulo the characteristic of $F$) by the latter to constrain the eigenvalues of $u$.
This does not work well in even characteristic, since the eigenvalues $1$ and $-1$ become equal in $F$ and therefore a lot of information is lost when reducing  $R[u]$-lattices to $F[u]$-modules. We remedy this by reducing to $\bar R[u]=R/4R[u]$ instead, but the price we pay is that we need to consider the actual arithmetic structure of the blocks rather than just the decomposition matrix.

\begin{definition}
    \begin{enumerate}
    \item Let $V_+$ denote the trivial $KC_2$-module, and $V_-$ the unique (up to isomorphism) non-trivial irreducible $KC_2$-module.
    \item Let $G_+$ and $G_-$ denote the (unique up to isomorphism) $RC_2$-lattices with $K$-span isomorphic to $V_+$ and $V_-$, respectively. Let $G_0$ denote the free $RC_2$-lattice of rank one.
    \item For an $RC_2$-lattice $X$ we denote by $\bar X$ the $\bar RC_2$-module obtained by reduction modulo~$4$. In particular, $\bar G_+$, $\bar G_-$ and $\bar G_0$ are the respective reductions of $G_+$, $G_-$ and $G_0$ modulo~$4$.
    \end{enumerate}
\end{definition}

It is well-known (e.g. by \cite[Corollary 1 after Theorem 2.2]{Gudivok}) that $G_+$, $G_-$ and $G_0$ are representatives for the isomorphism classes of indecomposable $RC_2$-lattices. While we do not have a classification of $\bar R C_2$-modules as such, we do know that an $\bar R C_2$-module that arises as the reduction modulo 4 of an $RC_2$-lattice is the sum of copies of $\bar G_+$, $\bar G_-$ and $\bar G_0$. This will be enough for our purposes. We will rely on the following, which is a special case of a theorem of Maranda \cite[Theorem (30.14)]{CurtisReinerI}]. Of course it also follows from the classification of $RC_2$-lattices. 

\begin{proposition}[Maranda]\label{prop maranda}
    Let $L$ and $L'$ be $RC_2$-lattices such that $L/4L\cong L'/4L'$. Then $L\cong L'$. \qed
\end{proposition}

\begin{lemma}\label{lemma ses lifts}
    Let $X$, $Y$ and $Z$ be $RC_2$-lattices and let
    \begin{equation}\label{ses over bar r}
        0\longrightarrow \bar X \stackrel{\iota}{\longrightarrow} \bar Y \stackrel{\varphi}{\longrightarrow} \bar Z \longrightarrow 0
    \end{equation}
    be a short exact sequence of $\bar R C_2$-modules. Then this sequence lifts to a short exact sequence of $RC_2$-lattices
    \begin{equation}
        0\longrightarrow X \longrightarrow Y \longrightarrow Z \longrightarrow 0.
    \end{equation}
\end{lemma}
\begin{proof}
    First we will show that we can assume without loss of generality that neither $\bar X$ nor $\bar Z$ has any direct summands isomorphic to $\bar G_0$. If $\bar Z$ has a summand isomorphic to $\bar G_0$, we can use projectivity to split off a summand of the form $0\longrightarrow 0 \longrightarrow \bar G_0 \longrightarrow \bar G_0 \longrightarrow 0$ from the sequence. 
    Note that $\bar X$, $\bar Y$ and $\bar Z$ are $\bar R$-free by assumption, which shows that our sequence is split as a sequence of $\bar R$-modules and therefore the functor $\Hom_{\bar R}(-,\bar R)$ takes it to another short exact sequence. We can view $\Hom_{\bar R}(-,\bar R)$ as a functor from $RC_2$-modules to $RC_2$-modules taking a representation to its inverse transpose. In particular, $\Hom_{\bar R}(\bar G_0, \bar R)\cong \bar G_0$, and applying $\Hom_{\bar R}(-,\bar R)$ twice is the same as applying the identity functor.
    If $\bar X$ has a direct summand isomorphic to $\bar G_0$, we can dualize the short exact by applying $\Hom_{\bar R}(-,\bar R)$, then split off a summand as above, and then apply $\Hom_{\bar R}(-,\bar R)$ again to dualize back.

    Next assume that $\bar Z=\bar Z_0\oplus \bar Z_1$ where $\bar Z_1$ is indecomposable, which means that $\bar Z_1$ is isomorphic to either $\bar G_+$ or $\bar G_-$. So $\bar Z_1/\rad(\bar Z_1)$ is simple, which shows that any homomorphism into $\bar Z_1$ is either surjective or has image contained in the radical. Hence there is a decomposition $\bar Y=\bar Y_0\oplus \bar Y_1$ such that $\bar Y_1$ is indecomposable and $\pi_{\bar Z_1}(\varphi(\bar Y_1))=\bar Z_1$, where $\pi_{\bar Z_1}$ denotes the projection onto $\bar Z_1$. If $\pi_{\bar Z_1}\circ \varphi$ induces an isomorphism between $\bar Y_1$ and $\bar Z_1$ then our short exact sequence decomposes as a direct sum
    \begin{equation}
    \left (        0\longrightarrow \bar X \stackrel{\iota}{\longrightarrow} \varphi^{-1}(\bar Z_0) \stackrel{\varphi}{\longrightarrow} \bar Z_0 \longrightarrow 0 \right )\oplus \left(         0\longrightarrow 0  \longrightarrow \bar Y_1 \stackrel{\varphi}{\longrightarrow} \varphi(\bar Y_1) \longrightarrow 0 \right).
    \end{equation}
    In this case both summands lift by induction on the rank of the terms. If $\pi_{\bar Z_1}\circ \varphi$  does not induce an isomorphism between $\bar Y_1$ and $\bar Z_1$, then $\bar Y_1$ must have bigger rank than $\bar Z_1$, which is only possible if $\bar Y_1\cong \bar G_0$.  If we decompose $\bar Y=\bar Y_0\oplus \bar Y_{\pm}$ where $\bar Y_0$ is a direct sum of copies of $\bar G_0$ and $\bar Y_\pm$ is a direct sum of copies of $\bar{G}_+$ and $\bar{G}_-$, then we can assume $\varphi(\bar Y_{\pm})\subseteq \rad(\bar Z)$ since otherwise we could split off further direct summands using the above argument. It follows that $\varphi(\bar Y_0)=\bar Z$, and, as a consequence,  $\pi_{\bar Y_\pm}(\iota(\bar X))=\bar Y_{\pm}$. 
    
    If $\bar Y_{\pm} = \bar Y_{\pm,0}\oplus \bar Y_{\pm,1}$ with $\bar Y_{\pm,1}$ indecomposable, then there is a direct sum decomposition $\bar X = \bar X_0\oplus \bar X_1$ with $\bar X_1$ indecomposable such that $\pi_{\bar Y_{\pm,1}}(\iota(\bar X_1))=\bar Y_{\pm,1}$. Since we assumed that $\bar X$ dos not have any summands isomorphic to $\bar G_0$ the map $\pi_{\bar Y_{\pm,1}}\circ \iota$ necessarily induces an isomorphism between $\bar X_1$ and $\bar Y_{\pm,1}$. We can now split our short exact sequence as 
    \begin{equation}
        \left (        0\longrightarrow \bar X_1 \stackrel{\iota}{\longrightarrow} \iota(\bar X_1){\longrightarrow} 0 \longrightarrow 0 \right )\oplus \left(         0\longrightarrow \iota^{-1}(\bar Y_{\pm,0}\oplus \bar Y_0)\longrightarrow \bar Y_{\pm,0}\oplus \bar Y_0  \stackrel{\varphi}{\longrightarrow} \bar Z \longrightarrow 0 \right).
    \end{equation}
    By repeated application of this step we can assume $\bar Y_\pm=0$. That is, we are left with a short exact sequence as in \eqref{ses over bar r} where $\bar Y$ is a direct sum of copies of $\bar G_0$, meaning $\bar Y$ is a free $\bar RC_2$-module. Now we can clearly lift $\varphi$ to a homomorphism $\hat \varphi:\ Y \longrightarrow Z$. This gives us a diagram as follows
    \begin{equation}
        \xymatrix{
            0\ar[r]& \ker(\hat \varphi) \ar[r] \ar@{->}[d] & Y \ar[r]^{\hat \varphi} \ar@{->>}[d]& Z\ar[r] \ar@{->>}[d]&0\\
            0\ar[r]&\bar X \ar[r]^{\iota} &\bar Y \ar[r]^{\varphi}& \bar Z\ar[r]&0,
        }
    \end{equation}
    and since these are short exact sequences of free $R$ and $\bar R$-modules, it follows by considering ranks that the leftmost vertical arrow induces and isomorphism between $\ker(\hat \varphi)/4\ker(\hat \varphi)$ and $\bar X$. By Proposition~\ref{prop maranda} we conclude 
    $\ker(\hat \varphi)\cong X$, making the top short exact sequence in the above diagram our desired lift.
\end{proof}

\begin{proposition}\label{prop:KleinFourDefect}
Assume that there is $2$-block of $G$ which has a Klein four defect group and decomposition matrix of the shape
\[\begin{pmatrix} 1 & 0 & 0 \\ 0 & 1 & 0 \\ 0 & 0 & 1 \\ 1 & 1 & 1 \\  \end{pmatrix}. \]
Let $p$ be any odd number. Label the irreducible characters belonging to the rows of this decomposition matrix $\psi,\chi_1,\chi_2$ and $\chi_3$ in that order. Let $u \in RG$ be a unit of order $2p$ and let $\zeta$ be a $p$-th root of unity such that $\mu(\zeta, u, \psi) = \mu(-\zeta, u, \psi) = 0$.  Then $\mu(\zeta, u, \chi_1+\chi_2) = \mu(\zeta, u, \chi_3)$.
\end{proposition}
\begin{proof}
    Assume without loss of generality that $R$ is big enough to define the corresponding block idempotent $b$ and that $RGb$ is split (see Remark~\ref{rem:Height0AndSplit} regarding the latter). 
    Define the idempotent 
    \begin{equation}
        e=\frac{1}{p}\sum_{i=0}^{p-1} u^{2i}\zeta^{-2i}
    \end{equation}
    so that for a $KG$-module $V$ the unit $u^2$ acts as $\zeta^2$ on $eV$, and therefore $u$ acts with eigenvalues $\pm\zeta$. In particular $e$ annihilates the irreducible $KG$-module corresponding to $\psi$, by the assumption $\mu(\zeta, u, \psi) = \mu(-\zeta, u, \psi) = 0$.
    
    So, $eRGbe$ is as described in Proposition~\ref{prop structure v4 block}. Let $W_i$ be a simple $KGb$-module corresponding to the character $\chi_i$ for every $1 \leq i \leq 3$. Then for each $i$ the $eKGbe$-module $V_i=eW_i$ is either simple or $0$ and $V_1$, $V_2$ and $V_3$ correspond to the matrix components in~\eqref{description order v4 block}. We can regard $L_1=R^a$, $L_2=R^b$ and $L_3=R^{a+b}$ as $eRGbe$-lattices within these modules by letting $eRGbe$ act by the respective matrix component of the order on the righthand side of the isomorphism in \eqref{description order v4 block}. If $\Delta_i$ for $i\in\{1,2,3\}$ denotes the representation of $eRGbe$ on $L_i$, then the description of $eRGbe$ in \eqref{description order v4 block} implies that
    \begin{equation}
        \Delta_3(a) \equiv \left( \begin{array}{cc} \Delta_1(a) & 0 \\ * & \Delta_2(a) \end{array} \right)\ (\textrm{mod }4) \quad \textrm{ for all $a\in eRGbe$.}
    \end{equation}
    That is, we get a short exact sequence of $eRGbe$-modules $0\longrightarrow L_2/4L_2\longrightarrow L_3/4L_3\longrightarrow L_1/4L_1\longrightarrow 0$. We can restrict this to a sequence of $R[u^p]\cong RC_2$-modules. Then Lemma~\ref{lemma ses lifts} implies that this lifts to a short exact sequence $0\longrightarrow L_2 \longrightarrow L_3\longrightarrow L_1\longrightarrow 0$ of $R[u^p]$-lattices (the maps in this sequence need no longer be $eRGbe$-homomorphisms). After tensoring with $K$ we get that $V_3\cong V_1\oplus V_2$ as $K[u^p]$-modules. The multiplicity of the trivial $K[u^p]$-module in $V_3$ is exactly $\mu(\zeta, u, \chi_4)$, and the multiplicity of the trivial module in $V_1 \oplus V_2$ is exactly $\mu(\zeta, u, \chi_2+\chi_3)$, so the claim follows.
\end{proof}

In the applications of this proposition given later in this article we will only need the following special form for principal blocks. However the general version above could also be useful in other situations, e.g. to study rational conjugation of torsion unit in $V(\mathbb{Z}G)$ where $G$ is a direct product of groups.

\begin{corollary}\label{cor:KleinFourDefect}
Assume that the Sylow $2$-subgroup of $G$ is a Klein four group and that the decomposition matrix of the principal $2$-block of $G$ is as in Proposition~\ref{prop:KleinFourDefect}.
Let $p$ be any odd number. Label the irreducible characters belonging to the rows of the decomposition matrix $\mathbf{1}, \chi_1,\chi_2$ and $\chi_3$ in that order, and assume that $\mathbf{1}$ is the trivial character. If $u \in V(\mathbb{Z}G)$ has order $2p$, then $\mu(\zeta_p, u, \chi_1+\chi_2) = \mu(\zeta_p, u, \chi_3)$.
\end{corollary}

\begin{proposition}\label{prop:D8Defect}
Assume that there is a $2$-block of $G$ with dihedral group of order $8$ and decomposition matrix of the shape
\[\begin{pmatrix} 1 & 0 \\ 1 & 0 \\ d & 1 \\ 1 & 1 \\ 1 & 1  \end{pmatrix}, \]
where $d\in \{0,2\}$. Let $p$ be any odd number and let $u\in RG$ be a unit of order $2p$. 

Let $\psi$ be the ordinary character corresponding to the first row of the decomposition matrix and denote by $\chi_1$ and $\chi_2$ the ordinary characters corresponding to the third and fourth row, respectively. If $\zeta$ is a $p$-th root of unity such that $\mu(\zeta, u, \psi) = \mu(-\zeta, u, \psi) = 0$, then $\mu(\zeta, u, \chi_1) = \mu(\zeta, u, \chi_2)$.
\end{proposition}
\begin{proof}
    As in the proof of Proposition~\ref{prop:KleinFourDefect} 
    we assume without loss of generality that $R$ is big enough to define the corresponding block idempotent $b$ and that $RGb$ is split. We again set $e=\frac{1}{p}\sum_{i=0}^{p-1} u^{2i}\zeta_p^{-2i}$ and we consider $eRGbe$. By Proposition~\ref{prop structure block d8} this $R$-order is isomorphic to $M_a(\Lambda)\subseteq M_a(K)\oplus M_a(K)\oplus M_a(K)$ for some $a\in \N$, where $\Lambda\subseteq K\oplus K \oplus K$ has an $R$-basis consisting of $(1,1,1)$, $(2,4,0)$ and $(4,0,0)$.

    Now $eRGbe$ acts on $L_1=R^a$, $L_2=R^a$ and $L_3=R^a$ through the projection of $M_a(\Lambda)$ to the first, second and third matrix component, respectively. We will restrict and consider $L_1$, $L_2$ and $L_3$ as lattices over $R[u^p]\cong RC_2$ only. In particular, lattice homomorphisms below will be $R[u^p]$-homomorphisms and may fail to be $eRGbe$-homomorphisms.
    Let us write $\Delta_i:\ R[u^p]\longrightarrow M_a(R)$ for the corresponding representations, and write 
    \begin{equation}
        \Delta_1(u^p) = A + 2B + 4C,\ \Delta_2(u^p)= A+4B, \ \Delta_3(u^p)=A \textrm{ for certain $A, B, C\in M_a(R)$}.
    \end{equation}
    It follows directly from the basis of $\Lambda$ given above that the representations must be of this form. By conjugating $u^p$ by an element of $\GL_a(R)\subset \GL_a(\Lambda)=\mathcal U(eRGbe)$ we can assume that
    \begin{equation}
        A=\left(\begin{array}{cccc} 0&I_f \\ I_f&0\\&&I_h \\&&&-I_l \end{array}\right)
    \end{equation}
    for $f,h,l\in \N_0$ such that $L_3 \cong G_0^f\oplus G_+^h \oplus G_-^l$ as an $R[u^p]$-lattice. Let $P\leq L_3$ denote the $R$-span of the first $2f$ standard basis vectors of $L_3=R^a$. Then $P$ spans a projective direct summand of $L_3$ as a lattice over $R[u^p]$. Since $\Delta_2$ and $\Delta_3$ coincide modulo $4$, the identity map from $\bar R^a$ to $\bar R^a$ induces an isomorphism between $L_3/4L_3$ and $L_2/4L_2$. Consider the restriction $\varphi: \ P/4P \longrightarrow L_2/4L_2$ of this isomorphism. Since $P$ is projective there exists an $R[u^p]$-homomorphism $\hat\varphi:\ P \longrightarrow L_2$ which lifts $\varphi$.  
    
    Now let $M$ be the matrix corresponding to the linear map from $R^a$ to $R^a$ which coincides with $\hat \varphi$ on the first $2f$ standard basis vectors and with the identity map on the remaining standard basis vectors.  Given that $\varphi$ was the restriction of the identity map on $\bar R^a$, the map $\hat \varphi$ will also coincide with the restriction of the identity map on $R^a$ modulo $4$ and therefore $M$ is congruent to the identity matrix modulo $4$. Write $M=I_a+4M'$. Now since $M$ induces an $R[u^p]$-homomorphism on the first $2f$ standard basis vectors, we have $\Delta_2(u^p)Mv = M\Delta_3(u^p)v$, if $v$ is one of these basis vectors. I.e.,
    \begin{equation}
        M^{-1}\Delta_2(u^p)M= \left(\begin{array}{cccc} 0&I_f \\ I_f&0\\&&I_h \\&&&-I_l \end{array}\right) + 4\cdot \left(\begin{array}{cccc} 0&0&*&* \\ 0&0&*&*\\0&0&*&* \\0&0&*&* \end{array}\right) 
    \end{equation}
    Now $(I_a + 2M', M, I_a) = (I_a+2M', I_a+4M', I_a)$ lies in $M_a(\Lambda)$ by the basis of $\Lambda$ exhibited above and is even invertible, as $2M'$ is an element of the radical. So, if we conjugate $u^p$ by this element we can assume that the first $2f$ columns in the matrix $B$ are equal to zero. So write 
    \begin{equation}
    B=\left(\begin{array}{cccc} 0 & 0 & * & * \\ 0 & 0 & * & *\\  0 & 0 & B_{33} & * \\  0 & 0 & * & B_{44} \\ \end{array}\right)
    \end{equation}
    where $B_{33} \in M_h(R)$ and $B_{44}\in M_{l}(R)$.

    Since all $\Delta_i$ are representations we have $A^2=I_a$ and $(A+4B)^2=I_a$. So we have $(A+4B)^2=I_a+4(AB+BA)+16B^2=I_a$, or $AB+BA=-4B^2$. Given the form of $A$ we get 
    \begin{equation}
        AB+BA = \left(\begin{array}{cccc}  0&0&*&*\\ 0&0&*&* \\ 0&0&2B_{33} &0\\0&0&0&-2B_{44} \end{array}\right).
    \end{equation} 
    It follows that both $B_{33}$ and $B_{44}$ are congruent to $0$ modulo $2$. If we now conjugate each $\Delta_i(u^p)$ by 
    \begin{equation}
        M''= \left(\begin{array}{ccc} I_{2f} \\ & 2I_h \\ &&I_l\end{array}\right)
    \end{equation} 
    we get representations $\Delta'_i$ which still take values in $M_a(R)$ (even though the conjugate of $B$ may have entries in $\frac{1}{2}R$) such that 
    \begin{equation}
        \Delta'_i(u^p)\equiv\left(\begin{array}{cccc} 0&I_f&*&* \\ I_f&0&*&*\\0&0&I_h&* \\0&0&0&-I_l \end{array}\right) \ (\textrm{mod }4).
    \end{equation}
    Let $L_i'$ denote the corresponding $R[u^p]$-lattices. The $L_i'$ correspond to full sublattices of the $L_i$, meaning they may be non-isomorphic to the $L_i$ but they have isomorphic $K$-span. The shape of $\Delta_i'$ implies that we have short exact sequences
    \begin{equation}
        0\longrightarrow \bar G_0^f  \longrightarrow L_i'/4L_i' \longrightarrow X_i \longrightarrow 0
    \end{equation}
    and 
    \begin{equation}
        0 \longrightarrow \bar G_+^h \longrightarrow X_i \longrightarrow \bar G_-^l\longrightarrow 0
    \end{equation}
    for some $\bar R$-free $\bar R[u^p]$-modules $X_i$ (corresponding to the $\bar R$-span of the last $h+l$ standard basis vectors). The first of these sequences is split since $\bar G_0^f$ is projective, and all terms are $\bar R$-free, so we can split off $\bar G_0^f$ after dualizing using $\Hom_{\bar R}(-,\bar R)$ like in Lemma~\ref{lemma ses lifts}. It therefore follows that $X_i$ is a direct sum of copies of $\bar G_0$, $\bar G_+$ and $\bar G_-$, which shows that $X_i$ is the reduction modulo $4$ of an $R[u^p]$-lattice. We can now apply Lemma~\ref{lemma ses lifts} to infer that the $K$-span of each $L_i'$ is isomorphic to the $K$-span of $G_0^f\oplus G_+^h\oplus G_-^l$. Notice that $\mu(\zeta, u, \chi_1)$ and $\mu(\zeta, u, \chi_2)$ are exactly the multiplicities of the trivial $K[u^p]$-module in the $K$-span of $L_1'$ and $L_2'$, respectively. As we just saw both are equal to $h$, so the claim follows.
\end{proof}

As in the Klein four case we now derive the consequence for the principal block which will be used in our applications.

\begin{corollary}\label{cor:D8Defect}
Assume that the principal $2$-block of $G$ has a dihedral defect group of order $8$ and that its decomposition matrix is as in Proposition~\ref{prop:D8Defect}. Let $p$ be any odd number and $u \in V(\mathbb{Z}G)$ of order $2p$. 
Denote by $\chi_1$ and $\chi_2$ the ordinary characters corresponding to the third and fourth row of the decomposition matrix, respectively. Then $\mu(\zeta_p, u, \chi_1) = \mu(\zeta_p, u, \chi_2)$.
\end{corollary}
\begin{proof}
To apply Proposition \ref{prop:D8Defect} it only remains to show that the first row of the decomposition matrix corresponds to the trivial character. This is clear by Remark~\ref{rem:Height0AndSplit} as there is exactly one ordinary irreducible character of height $1$ and the ordinary characters corresponding to the first two rows of the decomposition matrix have the same height.
\end{proof}

\section{Applications}

\subsection{General results}
 We write $\mathcal{C}_n$ for a set of representatives of the conjugacy classes of $G$ consisting of elements of order $n$. We will frequently use the fact that if $u \in V(\mathbb{Z}G)$ has order $n > 1$ and $\varepsilon_g(u) \neq 0$, then $g\neq 1$ and the order of $g$ divides $n$ (cf. Theorem~\ref{th:pAsofTorsionUnits}). Recall that for a prime $p$ an element of $G$ is called \emph{$p$-singular} if its order is divisible by $p$, and \emph{$p$-regular} if it is not.

\begin{lemma}\label{lem:FirstLemmaApplications}
Let $p$ and $q$ be different primes such that $G$ contains elements of order $p$ and $q$ but not of order $pq$. Let $\chi$ and $\psi$ be ordinary characters of $G$ which take the same values on all $q$-regular elements of $G$ and let $\xi$ be some $pq$-th root of unity (not necessarily primitive). If $u \in V(\mathbb{Z}G)$ is an element of order $pq$ and $\mu(\xi, u, \chi) = \mu(\xi, u, \psi)$ holds, then
 \begin{equation}\label{eqn main luthar passi lemma}
    \begin{array}{cl}&\displaystyle
 \Tr_{\mathbb{Q}(\zeta_q)/\mathbb{Q}}(\chi(u^p)\xi^{-p}) + \sum_{g \in \mathcal{C}_q} \varepsilon_g(u)\Tr_{\mathbb{Q}(\zeta_{pq})/\mathbb{Q}}(\chi(g)\xi^{-1}) \\
 =& \displaystyle \Tr_{\mathbb{Q}(\zeta_q)/\mathbb{Q}}(\psi(u^p)\xi^{-p}) + \sum_{g \in \mathcal{C}_q} \varepsilon_g(u)\Tr_{\mathbb{Q}(\zeta_{pq})/\mathbb{Q}}(\psi(g)\xi^{-1}) . 
    \end{array}
 \end{equation}
\end{lemma}

\begin{proof}
As $\chi$ and $\psi$ take the same values on all $q$-regular elements, we have in particular $\chi(1) = \psi(1)$. Moreover,
\[\chi(u^q) = \sum_{g \in \mathcal{C}_p} \varepsilon_g(u^q) \chi(g) =  \sum_{g \in \mathcal{C}_p} \varepsilon_g(u^q) \psi(g) = \psi(u^q). \]
By the Luthar-Passi formula from Proposition~\ref{pr:luthar-passi-multiplicity-formula} we obtain
\begin{equation}\label{eq:FirstLemmaApplication} 
\begin{array}{rcl}

\mu(\xi,u,\chi) &=& \frac{1}{pq}\left(\chi(1) + \Tr_{\mathbb{Q}(\zeta_p)/\mathbb{Q}}(\chi(u^q)\xi^{-q}) + \Tr_{\mathbb{Q}(\zeta_q)/\mathbb{Q}}(\chi(u^p)\xi^{-p}) + \Tr_{\mathbb{Q}(\zeta_{pq})/\mathbb{Q}}(\chi(u)\xi^{-1})\right) \vspace{3pt} \\
&=& \frac{1}{pq}\left(\chi(1) + \Tr_{\mathbb{Q}(\zeta_p)/\mathbb{Q}}(\chi(u^q)\xi^{-q}) + \Tr_{\mathbb{Q}(\zeta_q)/\mathbb{Q}}(\chi(u^p)\xi^{-p}) \right. \vspace{3pt}  \\&&\left.+ \Tr_{\mathbb{Q}(\zeta_{pq})/\mathbb{Q}}\left( \sum_{g \in \mathcal{C}_p}\varepsilon_g(u)\chi(g)\xi^{-1} \right) 
 + \Tr_{\mathbb{Q}(\zeta_{pq})/\mathbb{Q}}\left(\sum_{g \in \mathcal{C}_q}\varepsilon_g(u)\chi(g)\xi^{-1}\right)\right)
\end{array}
\end{equation}
and the analogous equation for $\mu(\xi, u, \psi)$. So the assumed equation $\mu(\xi, u,\chi) = \mu(\xi, u, \psi)$ simplifies to the claimed equation~\eqref{eqn main luthar passi lemma} after cancelling equal terms on both sides, using $\chi(u^q) = \psi(u^q)$ and the fact that $\chi(g) = \psi(g)$ for all $g \in \mathcal{C}_p$.	
\end{proof}

As an application of the lemma we obtain the following.

\begin{corollary}\label{cor:ChiYPsiY}
Assume the situation of Lemma~\ref{lem:FirstLemmaApplications} and, in addition, that $G$ contains exactly one conjugacy class of elements of order $q$. Let $y \in G$ be an element in this class. Then $\chi(y) = \psi(y)$.
\end{corollary}

\begin{proof}
The condition that there is only one conjugacy class of elements of order $q$ implies that if $\varepsilon_g(u^p) \neq 0$, then $g \in y^G$. In particular, $\chi(y) = \chi(u^p)$ and $\psi(u^p) = \psi(y)$. Moreover, we can set $\mathcal{C}_q = \{y \}$. Observe that since $y^G$ is a rational class by assumption the values of $\chi(y)$ and $\psi(y)$ are both rational, which implies $\Tr_{\mathbb{Q}(\zeta_{pq})/\mathbb{Q}}(\chi(y)\xi^{-1})  = \chi(y) \Tr_{\mathbb{Q}(\zeta_{pq})/\mathbb{Q}}(\xi^{-1})$.
So equation~\eqref{eqn main luthar passi lemma}
 from Lemma~\ref{lem:FirstLemmaApplications} becomes
 \begin{equation}
 \begin{array}{rl}
 & \chi(y)\Tr_{\mathbb{Q}(\zeta_q)/\mathbb{Q}}(\xi^{-p}) + \varepsilon_y(u)\chi(y)\Tr_{\mathbb{Q}(\zeta_{pq})/\mathbb{Q}}(\xi^{-1}) \vspace*{.1cm} \\ 
 =& \psi(y)\Tr_{\mathbb{Q}(\zeta_q)/\mathbb{Q}}(\xi^{-p}) + \varepsilon_y(u)\psi(y)\Tr_{\mathbb{Q}(\zeta_{pq})/\mathbb{Q}}(\xi^{-1}) .
 \end{array} 
 \end{equation}
This can be rearranged as
\begin{align}\label{eq:ChiYPsiY}
(\chi(y) - \psi(y))(1 + \varepsilon_y(u)r) = 0,\quad  \textrm{ where } r = \frac{\Tr_{\mathbb{Q}(\zeta_{pq})/\mathbb{Q}}(\xi^{-1})}{\Tr_{\mathbb{Q}(\zeta_q)/\mathbb{Q}}(\xi^{-p})}.
\end{align}
Now $\varepsilon_y(u) \equiv 0 \ (\textrm{mod }q)$ by Lemma~\ref{lem:Congruence}. Notice that, regardless of the choice of $\xi$, we have $\Tr_{\mathbb{Q}(\zeta_q)/\mathbb{Q}}(\xi^{-p}) \equiv -1\ (\textrm{mod }q)$, since the only possible values of the trace are $q-1$ and $-1$. So $\varepsilon_y(u)r \equiv 0 \ (\textrm{mod }q)$ and therefore $1+\varepsilon_y(u)r \neq 0$. So, $\chi(y) - \psi(y) = 0$ holds by \eqref{eq:ChiYPsiY}.
\end{proof}

We are now ready to apply our results on the units in $2$-blocks from the previous section.

\begin{proposition}
Assume that a Sylow $2$-subgroup of $G$ is a Klein four group and that the decomposition matrix of the principal $2$-block of $G$ has shape
\[\begin{pmatrix} 1 & 0 & 0 \\ 0 & 1 & 0 \\ 0 & 0 & 1 \\ 1 & 1 & 1 \\  \end{pmatrix}. \]
Let $p$ be any prime. If $V(\mathbb{Z}G)$ contains elements of order $2p$, then $G$ contains elements of order $2p$.
\end{proposition}

\begin{proof}
The case $p=2$ follows from Lemma~\ref{lem:CohnLivingstone}, so we assume $p$ odd from now on.
We first show that $G$ contains only one conjugacy class of involutions. Indeed, let $P$ be a Sylow $2$-subgroup of $G$ and assume that not all the non-trivial elements of $P$ are conjugate in $G$. Then any element of $G$ normalizing $P$ must centralize one of the non-trivial elements. But it  cannot interchange the other two non-trivial elements either, as there is no involution contained in $N_G(P)$ which is not in $P$ (since $P$ is a Sylow $2$-subgroup of $G$). So the centralizer of $P$ in $G$ equals the normalizer and by the $p$-complement theorem of Burnside \cite[IV, Hauptsatz 2.6]{HuppertI} the group $G$ contains a normal $2$-complement. But then $G$ has a non-trivial linear character whose restrictions to all $2$-regular elements is trivial, contradicting the assumed shape of the decomposition matrix (which has no repeated rows). Hence $G$ contains exactly one conjugacy class of involutions. We denote a fixed involution from this class by $y$.

Now define $\chi$ as the sum of the three irreducible ordinary characters of the principal block whose restriction to $2$-regular elements is an irreducible Brauer character (i.e. $\chi$ corresponds to the sum of the first three rows of the decomposition matrix) and let $\psi$ denote the fourth irreducible ordinary character in the block (corresponding to the last row of the decomposition matrix). Assume that a unit $u \in V(\mathbb{Z}G)$ of order $2p$ exists. We will show that this assumption implies $\chi = \psi$, contradicting the linear independence of ordinary irreducible characters.

First note that $\chi(g) = \psi(g)$ holds for any $2$-regular element $g \in G$ by the shape of the decomposition matrix. 
Moreover, $\mu(\zeta_p, u, \chi) = \mu(\zeta_p, u, \psi)$ holds by Corollary~\ref{cor:KleinFourDefect} and, as noted before, $G$ contains exactly one conjugacy class of involutions. Hence we can apply Corollary~\ref{cor:ChiYPsiY} to conclude $\chi(y) = \psi(y)$. By \cite[Theorem 7.7]{Navarro} the irreducible ordinary characters in a principal block of Klein four defect for a group with only one conjugacy class of involutions take the same value on all $2$-singular elements. So if $g \in G$ is any element of even order then 
$\chi(g) = \chi(y) = \psi(y) = \psi(g)$, yielding the desired equality $\chi = \psi$ giving the contradiction.
\end{proof}

\begin{proposition}\label{prop:ApplicationD8}
Assume that a Sylow $2$-subgroup of $G$ is a dihedral group of order $8$ and that the decomposition matrix of the principal $2$-block has shape
\[\begin{pmatrix} 1 & 0 \\ 1 & 0 \\ d & 1 \\ 1 & 1 \\ 1 & 1  \end{pmatrix}, \]
where $d \in \{0,2 \}$. 
Let $p$ be any prime. Then $V(\mathbb{Z}G)$ contains elements of order $2p$ if and only if $G$ contains elements of order $2p$.
\end{proposition}
\begin{proof}
The case $p=2$ follows from Lemma~\ref{lem:CohnLivingstone}, so we assume $p$ odd from now on. Note that we can assume that the first row of the decomposition matrix corresponds to the trivial representation by Remark~\ref{rem:Height0AndSplit}.
We first observe that $G$ contains exactly two conjugacy classes of involutions. Indeed, three classes would imply the existence of a normal $2$-complement by Burnside's $p$-complement theorem \cite[IV, Hauptsatz 2.6]{HuppertI}, which in turn would mean that there are strictly more than two irreducible ordinary characters of $G$ whose restriction to the $2$-regular elements is trivial. Moreover, there are exactly two ordinary irreducible characters whose restriction to $2$-regular elements is trivial, so $G$ has a normal subgroup $N$ of index $2$. It follows that if $P$ is a Sylow $2$-subgroup of $G$, then $P/(P\cap N) \cong C_2$, meaning that $N$ cannot contain all five involutions of $P$. Let $x$ and $y$ denote representatives for the two classes of involutions, where $y$ does not lie in $N$.

Assume now that $u \in V(\mathbb{Z}G)$ of order $2p$ exists. We claim that $\varepsilon_x(u^p) = 1$ or $\varepsilon_y(u^p) = 1$. Indeed, $\varepsilon_x(u^p) + \varepsilon_y(u^p)=1$ follows by Theorem~\ref{th:pAsofTorsionUnits} and if $\eta$ is the non-trivial ordinary linear character  in the principal block then 
\[\pm 1 = \eta(u^p) = \varepsilon_x(u^p) - \varepsilon_y(u^p).\]
Together the two equations imply the claim. We assume that $\varepsilon_y(u^p) = 1$, swapping $x$ and $y$ if necessary, as we will not need the assumption $y \notin N$ any more.

Next denote by $\chi$ the irreducible character corresponding to the third row of the decomposition matrix and by $\psi$ the character corresponding to the fourth row. Then Corollary~\ref{cor:D8Defect} implies $\mu(\zeta_p, u, \chi) = \mu(\zeta_p, u, \psi)$. Hence by Lemma~\ref{lem:FirstLemmaApplications} we obtain, with $\xi = \zeta_p$ and $q=2$, that
\[ \chi(y) - (\varepsilon_y(u)\chi(y) + \varepsilon_x(u)\chi(x)) = \psi(y) - (\varepsilon_y(u)\psi(y) + \varepsilon_x(u)\psi(x)) \]
and thus
\begin{align}\label{eq:ApplicationLemmaForD8}
0 = (\chi(y)-\psi(y))(1-\varepsilon_y(u)) - (\chi(x)-\psi(x))\varepsilon_x(u).
\end{align}

Denote by $\varphi$ the non-trivial irreducible $2$-Brauer character in the block, which corresponds to the second column of the decomposition matrix. Then from the shape of the decomposition matrix we get $\chi(1) = \varphi(1) + d$ and $\psi(1) = \varphi(1) + 1$. Hence, $\chi(1) - \psi(1) = d - 1$.
Note that if $g \in G$ is an involution and $\alpha$ any ordinary character of $G$, then $\alpha(1) \equiv \alpha(g)\ (\textrm{mod } 2)$, just because a matrix representing $g$ has eigenvalues $1$ or $-1$. So, by the assumption that $d\in\{0,2\}$, we obtain
\[\chi(y) -\psi(y) \equiv \chi(x) - \psi(x) \equiv d - 1 \equiv 1 \ (\textrm{mod } 2). \]
Substituting these congruences in \eqref{eq:ApplicationLemmaForD8} we thus get 
\[0 \equiv 1 - \varepsilon_y(u) - \varepsilon_x(u)  \ (\textrm{mod } 2),\]
contradicting the fact that $\varepsilon_y(u) + \varepsilon_x(u)$ is even by Lemma~\ref{lem:Congruence}.
\end{proof}

\subsection{Applications for $\PSL(2,p^f)$}

We first summarize some well know facts about almost simple groups with socle $\PSL(2,p^f)$.

\begin{proposition}\label{prop:OutPSL2}\cite[II, Hauptsatz 8.27]{HuppertI}, \cite[Section 3.3.4]{Wilson}
Let $G = \PSL(2,p^f)$ where $f$ is odd and let $d = \gcd(2,p-1)$. The order of $G$ is $\frac{(p^f-1)p^f(p^f+1)}{d}$ and the Sylow $2$-subgroups of $G$ are dihedral. The Sylow $2$-subgroups of $\PGL(2,p^f)$ are also dihedral. The outer automorphism group of $G$ is isomorphic to $C_d \times C_f$. Here the direct factor decomposition can be a realized such that the cyclic group of order $f$ is generated by the automorphism induced by the entry-wise application of the Frobenius automorphism of $\mathbb{F}_{p^f}$ to the entries of matrices in $\SL(2,p^f)$ and such that moreover, if $\alpha$ is a generator of the cyclic factor $C_d$ in the outer automorphism group, then $\langle G, \alpha \rangle \cong \PGL(2,p^f)$.
\end{proposition}

Now we are ready to apply our previous results.

\begin{theorem}\label{th:PSLOrder2p}
Let $q=p^f$ be a prime power such that $q \equiv \pm 3\ (\textrm{mod } 8)$. If $G$ is an almost simple group with socle $\PSL(2,q)$, then $V(\mathbb{Z}G)$ contains no elements of order $2p$.
\end{theorem}
\begin{proof}
It suffices to show the claim for the maximal group among the almost simple groups with socle $\PSL(2,q)$, that is, for $\Aut(\PSL(2,q))$. So set $G = \Aut(\PSL(2,q))$ and recall that by Proposition~\ref{prop:OutPSL2} we have $\Out(\PSL(2,p^f)) \cong C_2 \times C_f$. Moreover, let $\delta \in \{ \pm 1\}$ so that $\delta \equiv q\ (\textrm{mod } 4)$. Note that by the assumption on $q=p^f$ the exponent $f$ is odd.

By the assumption on $q=p^f$ a Sylow $2$-subgroup of $\PSL(2,q)$ has order $4$ and hence is a Klein four group by Proposition~\ref{prop:OutPSL2}.
The decomposition matrix of the principal $2$-block of $\PSL(2,q)$ is either  
\begin{center}
$\begin{pmatrix} 1 & 0 & 0 \\ 0 & 1 & 0 \\ 0 & 0 & 1 \\ 1 & 1 & 1  \end{pmatrix}$ \ if  $q \equiv 3\ (\textrm{mod } 8)$, \ or $\begin{pmatrix} 1 & 0 & 0 \\ 1 & 1 & 0 \\ 1 & 0 & 1 \\ 1 & 1 & 1 \end{pmatrix}$ \ if $q \equiv -3\ (\textrm{mod } 8)$, by \cite[VIII]{Burkhardt}.
\end{center}
Here the characters in the second and third row have degree $\frac{q+\delta}{2}$ and the character in the fourth row is the Steinberg character of degree $q$. By \cite[Lemmas 4.4, 4.5]{White} all these characters are invariant under the outer automorphism of order $f$, while the two characters of degree $\frac{q+\delta}{2}$ form an orbit under the action of the outer automorphism of order $2$. 
This implies that the decomposition matrix of the principal $2$-block of $G$ is the same as that for $\PGL(2,q)$, which is
\begin{center}
$\begin{pmatrix} 1 & 0 \\ 1 & 0 \\ 0 & 1 \\ 1 & 1 \\ 1 & 1  \end{pmatrix}$ \ if  $q \equiv 3 \ (\textrm{mod } 8)$, \ or $\begin{pmatrix} 1 & 0 \\ 1 & 0 \\ 2 & 1 \\ 1 & 1 \\ 1 & 1  \end{pmatrix}$ \ if $q \equiv -3 \ (\textrm{mod } 8)$.
\end{center}
By Proposition~\ref{prop:OutPSL2} the Sylow 2-subgroup of $\PGL(2,q)$ is dihedral of order $8$ and hence so is the Sylow 2-subgroup of $G$. The theorem hence follows from Proposition~\ref{prop:ApplicationD8}.
\end{proof}

We are now ready to answer the Prime Graph Question for new classes of almost simple groups by proving Theorem~\ref{th:ApplicationPQ}.

\begin{proof}[Proof of Theorem~\ref{th:ApplicationPQ}]
The case where $q$ is a prime is included in Proposition~\ref{prop:PSL2KnwonStuff}. So let us assume that $q=p^f$ for $f\geq 2$ and that the odd part of $(q-1)(q+1)$ is square free. We note that this implies that $f$ is comprime to $(p^f-1)(p^f+1)$. Indeed, if $r$ is an odd prime that divides $p^f-1$ and $f$, say $f = rm$, then $p^f \equiv 1 \ (\textrm{mod } r)$ implies $p^{rm} \equiv 1 \ (\textrm{mod } r)$. But this implies $p^m \equiv 1 \ (\textrm{mod } r)$ and therefore $p^{rm}-1 \equiv 0 \ (\textrm{mod } r^2)$, so that the odd part of $p^f-1$ is not square free. Similarly $p^f \equiv -1 \ (\textrm{mod } r)$ and $r \mid f$ implies $p^f \equiv -1 \ (\textrm{mod } r^2)$. 

Recall that $\PSL(2,q)$ has order $\frac{(q-1)q(q+1)}{2}$ and that the outer automorphism group has order $2f$ where $q=p^f$ (cf. Proposition~\ref{prop:OutPSL2}). Now assume $u \in V(\mathbb{Z}G)$ is of order $rs$ for some primes $r$ and $s$. Consider first the case where both $r$ and $s$ are divisors of the order of $\PSL(2,q)$. If one of them, say $r$, is odd but does not equal $p$, then a Sylow $r$-subgroup of $G$ has order $r$ by the assumptions and the fact that $r$ does not divide $f$. So $G$ contains an element of order $rs$ by Theorem~\ref{th:CaicedoMargolis}. If on the other hand $rs = 2p$, then $u$ cannot exist by Theorem~\ref{th:PSLOrder2p}. 
Finally, consider the case that one of $r$ or $s$ does not divide the order of $\PSL(2,q)$. Then $G$ contains an element of order $rs$ by Proposition~\ref{prop:KimmerleKonovalovQuotient}. Overall, the Prime Graph Question has a positive answer for $G$.
\end{proof}

We finish with a corollary improving a previous result:

\begin{corollary}
Assume that  $G$ is an almost simple group whose order is divisible by at most four pairwise distinct primes and that the Prime Graph Question has a negative answer for $G$. Then the socle of $G$ is in the following finite list:
\[\PSL(2,81), \ \PSL(2,243), \ \PSL(3,7), \ \PSL(3,8), \ \PSL(3,17), \ \operatorname{PSp}(4,7), \ \operatorname{Sz}(32). \]
\end{corollary}
\begin{proof}
The Prime Graph Question for groups of order divisible by at most four pairwise different primes was investigated in \cite{BachleMargolis4primaryI, BachleMargolis4primaryII}. In particular, by \cite[Theorem A]{BachleMargolis4primaryII} to prove the claim of the corollary it suffices to answer the Prime Graph Question for almost simple groups with socle $\PSL(2,3^f)$ whose order is divisible by at most four different primes and such that $f=3$ or $f \geq 7$. If $f=3$, the Prime Graph Question has a positive answer by Theorem~\ref{th:ApplicationPQ}, so assume $f \geq 7$. In this critical case by \cite[Theorem B]{BugeaudCaoMignotte} both $\frac{3^f-1}{2}$ and $\frac{3^f+1}{4}$ are prime. In particular, $f$ is odd and $3^f$ is hence congruent to $\pm 3$ modulo $8$. Also the odd part of $(3^f-1)(3^f+1)$ is square free. So we can apply Theorem~\ref{th:ApplicationPQ}.
\end{proof}

\begin{remark}
Theorem~\ref{th:ApplicationPQ} in particular answers the Prime Graph Question for $\PSL(2,27)$, the smallest simple group for which it had been open, a fact highlighted already at the end of \cite{HertweckHoefertKimmerle}. The smallest group for which the Prime Graph Question is open now is $\PSL(2,64)$. More related to the groups studied here the question remains open for $\PSL(2,81)$ which has a dihedral group of order 16 as a Sylow $2$-subgroup. In both cases units of order $6$ are critical.
\end{remark}

%

\section*{Funding}
This work was supported by the Madrid Government (Comunidad de Madrid - Spain) under the multiannual Agreement with UAM in the line for the Excellence of the University Research Staff in the context of the V PRICIT (Regional Programme of Research and Technological Innovation). The authors acknowledge financial support from the Spanish Ministry of Science and Innovation, through the Severo Ochoa Programme for Centers of Excellence in R\&D (CEX2019-000904-S) and through the Ramon y Cajal grant programme.

\bibliographystyle{amsalpha}
\bibliography{references}
\end{document}